\documentclass[a4paper]{article}

\usepackage{amsmath}
\usepackage{amsthm}
\usepackage[dvips]{graphicx}
\usepackage{epsfig}
\usepackage{float}
\usepackage{amssymb}
\usepackage{latexsym}
\usepackage{color}

\makeatletter
\newcommand\footnoteref[1]{\protected@xdef\@thefnmark{\ref{#1}}\@footnotemark}
\makeatother

\newtheorem{Theorem}{Theorem}
\newtheorem{Proposition}{Proposition}
\newtheorem{Lemma}{Lemma}
\newtheorem{Corollary}{Corollary}

\newtheorem*{rep@theorem}{\rep@title}
\newcommand{\newreptheorem}[2]{%
\newenvironment{rep#1}[1]{%
 \def\rep@title{#2 \ref{##1}}%
 \begin{rep@theorem}}%
 {\end{rep@theorem}}}
\makeatother

\newreptheorem{theorem}{Theorem}
\newreptheorem{corollary}{Corollary}

\theoremstyle{definition}
\newtheorem{definition}{Definition}
\newtheorem{remark}{Remark}
\newtheorem{question}{Question}

\def\ZZ{{\mathbb Z}}

\def\S{{\mathbf S}}

\begin{document}

\title{Double branched covers of tunnel number one knots}

\author{Yeonhee Jang and Luisa Paoluzzi}

\date{\today}
\maketitle

\begin{abstract}
\vskip 2mm

We provide criteria ensuring that a tunnel number one knot $K$ is not
determined by its double branched cover, in the sense that the double
branched cover is also the double branched cover of a knot $K'$ not equivalent
to $K$. 

\vskip 2mm


\vskip 2mm


\end{abstract}


\section{Introduction}
\label{s:intro}


A knot in the $3$-sphere is said to have \emph{tunnel number one} if its
exterior admits a Heegaard splitting of genus $2$. This is equivalent to say
that there is a properly embedded arc in the knot exterior whose complement
is a genus $2$ handlebody. Torus knots, $2$-bridge knots, and $(1,1)$-knots,
that is knots that can be put in $1$-bridge position with respect to a genus 
$1$ Heegaard splitting of the $3$-sphere, are all tunnel number one knots 
(see \cite{BM, BRZ, K3, MS} for instance). 
There is a vast literature studying different aspects of tunnel number one
knots and it appears that they are very special in many regards. For instance,
this class does not contain any composite knots \cite{N} or any \emph{Conway 
reducible} knots \cite{S}. Recall that a knot is \emph{Conway reducible} if it 
admits a \emph{Conway sphere}, that is an essential four-holed sphere properly 
embedded in the exterior of the knot. More generally, the exterior of a tunnel 
number one knot does not contain any planar meridinal essential surface (see, 
\cite{GR}). Having tunnel number one seems to often be a quite constraining 
condition. For instance, tunnel number one knots that are satellites were 
classified by Morimoto and Sakuma in \cite{MS}: these are a very restricted 
family of knots whose exteriors are obtained by gluing together the exteriors 
of a $2$-bridge link and of a torus knot. Similarly, alternating knots having 
tunnel number one were also classified (see \cite{L}) and, again, only very 
specific knots, namely $2$-bridge knots and certain Montesinos knots with three 
tangles, are in this class. On the other hand, there are hyperbolic tunnel 
number one knots with arbitrarily large bridge number (see \cite{J1}) and even 
genus one bridge number (see, \cite{JT,BTZ}). Note that, since this class 
contains $2$-bridge knots, it also contains hyperbolic knots with arbitrarily 
large volume.

The aim of this work is to unveil the specificity of these knots with respect to
their double branched covers. We need a definition.

\begin{definition}
We say that a knot $K$ is \emph{not determined by its double branched cover}
$M(K,2)$ if there is a knot $K'$ not equivalent to $K$ such that the manifolds 
$M(K',2)$ and $M(K,2)$ are homeomorphic. In this case we say that $K'$ is a
\emph{$2$-twin of $K$}.
\end{definition}

Throughout this paper, the {\it double branched cover of a knot} means the double cover of the $3$-sphere branched along the knot.
We have the following result.

\begin{Theorem}\label{t:main}
Let $K$ be a tunnel number one knot with bridge number $b\ge5$, then $K$
is not determined by its double branched cover. Moreover, if the double
branched cover $M=M(K,2)$ of $K$ has Heegaard genus $2$ then either $b=3$
(and so $K$ is a $(1,1)$-knot) or $K$ is not determined by $M$. In particular a 
$(1,1)$-knot of bridge index $\ge 4$ is not determined by its double branched 
cover.
\end{Theorem}

An immediate consequence is that a tunnel number one knot is not determined by
its double branched cover if the Heegaard genus of the cover is two and the
knot is not $(1,1)$. Surprisingly enough we know no examples of knots with this
property, so it is natural to ask the following:

\begin{question}
Is there a tunnel number one knot that is not $(1,1)$ but whose double branched 
cover has Heegaard genus 2?
\end{question}

The peculiarity of tunnel number one knots expressed in Theorem \ref{t:main} is 
related to another special feature of these knots, that is the fact that they 
are all \emph{strongly invertible}. Recall that a knot $K$ is 
\emph{strongly invertible} if there is an orientation-preserving involution of 
the $3$-sphere with non-empty fixed-point set which leaves $K$ invariant and 
reverses its orientation, in particular the fixed-point set of the involution 
meets the knot in precisely two points. Such an involution is called a 
\emph{strong inversion}. Observe that the existence of a strong inversion for a 
tunnel number one knot is a straightforward consequence of the fact that there 
is a hyperelliptic involution of a genus-two Heegaard surface which extends to 
a global involution of the genus two splitting (\cite{BH}). 

It is well-known that the presence of symmetries may sometimes reflect the fact
that the knot is not determined by its cyclic-branched covers (see, for
instance \cite{Na,S1}). For a $\pi$-hyperbolic knot $K$, it was initially 
observed by Boileau and Flapan in \cite{BF} that if $K$ is not determined by 
its double branched cover then it is either strongly invertible or admits a 
\emph{$\pi$-rotation}, that is an orientation-preserving involution of the 
$3$-sphere with non-empty fixed-point set which leaves $K$ invariant and 
preserves its orientation. There are, however, knots that are not determined by 
their double covers and admit no symmetry: this is the case for instance of 
certain Conway reducible knots which admit $2$-twins obtained by Conway 
mutation; note that for some Conway reducible knots, like Montesinos knots 
with at least four tangles, the presence of symmetries is not related to the
existence of $2$-twins (see \cite{V,M,C,KT} and \cite{P,MW} for other types of 
examples).

Observe that, in the case of tunnel number one knots, the trivial knot is
determined by its double branched cover by a result of Waldhausen \cite{W} (and 
more generally by the positive solution to the Smith conjecture \cite{MB}) as 
are $2$-bridge knots, by a work of Hodgson and Rubinstein \cite{HR}.

The only cases that remain to be considered are those of tunnel number one
knots of bridge number $3$ and $4$. Here we will consider several classes of
tunnel number one knots and show that, even if the bridge index is $3$ or $4$, 
these knots are often not determined by their double branched covers. The 
following is mostly well-known and covers all non $\pi$-hyperbolic tunnel 
number one knots.

\begin{Proposition}\label{p:cases}
\begin{itemize}
\item A torus knot $T(p,q)$, $p<q$, is determined by its double branched cover 
if and only if either $p=2$, that is it is a $2$-bridge knot, or 
$(p,q)\in\{(3,4),(3,5)\}$.
\item A tunnel number one Montesinos knot is not determined by its double
branched cover if and only if it is not a $2$-bridge knot or a torus knot, and
admits a torus knot as a $2$-twin.
\item A satellite tunnel number one knot $K$ is never determined by its double
branched cover. $K$ admits a $2$-twin $K'$ which is hyperbolic and Conway 
reducible, in particular $K'$ is not a tunnel number one knot. $K$ does not admit any
tunnel number one $2$-twins.
\end{itemize}
\end{Proposition}

Taking into account the classification provided by Klimenko and Sakuma
\cite{KS}, the second part of the result above can be made more precise. The 
statement is however slightly involved and will be given in 
Section~\ref{s:easycases}.

Using Lackenby's classification of alternating tunnel number one knots \cite{L}
we get the following.

\begin{Corollary}\label{c:alternating}
Tunnel number one alternating knots are determined by their double branched
covers.
\end{Corollary} 

It follows that alternating tunnel number one knots trivially fulfill Greene's
conjecture that the $2$-twin of an alternating prime knot is alternating
\cite{G}.

Unfortunately, for $\pi$-hyperbolic tunnel number one knots providing a 
complete answer seems hard. Nonetheless, we will see in Section~\ref{s:two-pi} 
that the knots belonging to two families, namely twisted torus knots and knots 
obtained by Dehn surgery on a component of a hyperbolic $2$-bridge link, are 
``generically" not determined by their double branched covers. Note that the 
arguments follow the same lines as those used by Reni and Zimmermann in 
\cite{RZ} to provide examples of knots not determined by their double branched covers.  

More interestingly, we show that a $\pi$-hyperbolic tunnel number one knot of
bridge index three or four can only be determined by its double branched cover 
if its minimal genus Heegaard splittings have small Hempel distance.  

\begin{Theorem}\label{t:hempel}
Let $K$ be a $\pi$-hyperbolic tunnel number one knot with bridge number 
$b\in\{3,4\}$ and let $M$ be its double branched cover.
Let $g\in\{2,3\}$ be the Heegaard genus of $M$ and assume that any minimal 
genus Heegaard splitting for $M$ has Hempel distance at least $2g+1$. Then $K$ 
cannot be determined by $M$.
\end{Theorem}

To be more precise, we only require that a specific minimal genus Heegaard
splitting induced either by a $(1,1)$-presentation for the knot or a genus-two
splitting of its exterior satisfies the given condition on the Hempel distance.

We remark that Theorem~\ref{t:hempel} does not mean that a ($\pi$-hyperbolic tunnel number one) knot is determined by its double branched cover $M$ if $M$ admits a minimal genus Heegaard splitting with small Hempel distance.
See Remark~\ref{r:smalldistance} for examples of knots which admit 2-twins and whose double branched covers admit genus-$2$ Heegaard splittings with Hempel distance at most $2$.

\bigskip

The paper is organised as follows: In Section~\ref{s:prfthm} we prove
Theorem~\ref{t:main}. In Section~\ref{s:easycases} we provide a 
characterisation of the non $\pi$-hyperbolic tunnel number one knots that are 
not determined by their double branched covers. In the last two sections we 
discuss the case of $\pi$-hyperbolic tunnel number one knots: in
Section~\ref{s:two-pi} we show that knots of this type belonging to two 
specific classes are generically not determined by their double branched covers 
and in Section~\ref{s:hempel} we prove Theorem~\ref{t:hempel}.


\section{Proof of Theorem~\ref{t:main}}
\label{s:prfthm}


Let $K$ be a tunnel number one knot and $M=M(K,2)$ its double branched cover.
Let $V$ be a tubular neighborhood of the union of $K$ and an unknotting tunnel 
$\tau$, and $W$ be the closure of $\S^3\setminus V$ (see 
Figure~\ref{fig_strong_involution}).
%
\begin{figure}[btp]
\begin{center}
\includegraphics*[width=10cm]{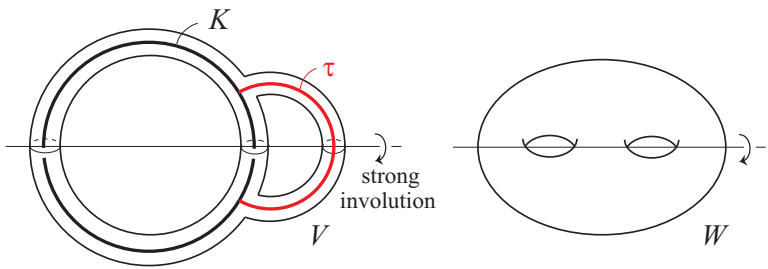}
\end{center}
\caption{$V$, $W$ and a strong involution.}
\label{fig_strong_involution}
\end{figure}
%
Then ${\rm Cl}(V\setminus N(K))\cup W$ is a Heegaard splitting of genus $2$ of 
the exterior of $K$. Consider the lift to $M$ of the Heegaard splitting 
${\rm Cl}(V\setminus N(K))\cup W$: it is a Heegaard splitting of genus $3$ of 
$M$. This fact can be seen as follows. The $2$-fold cover of $V$ branched along 
$K$ is the union of a solid torus which is the neighborhood of the lift of $K$ 
and two $1$-handles corresponding to the lifts of $\tau$. Remark that the cover 
is unbranched on the complement. It follows that the Heegaard genus of $M$ is 
at most $3$. Denote by $s$ the involution of $M$ that generates the group of 
deck transformations of the cover, so that we have 
$(M,Fix(s))/\langle s \rangle=(\S^3,K)$.

Consider now the strong inversion of $K$ that acts as a hyperelliptic involution
of each handlebody of the genus-two splitting and whose fixed-point set meets 
the chosen tunnel for $K$ in one point. Let $u$ and $su$ be the two lifts of 
the strong inversion to $M$. We claim that one of these two lifts, say $u$, 
acts as a hyperelliptic involution of the genus-three Heegaard splitting of 
$M$. First of all note that both $Fix(u)$ and $Fix(su)$ are non empty and, 
since $M$ is a $\ZZ/2\ZZ$-homology sphere, connected. Since $s$ commutes with 
both $u$ and $su$ it must leave both their fixed-point sets invariant. Now the 
union of $Fix(u)$ and $Fix(su)$ must meet each of the lifts of the tunnel. 
Since $s$ exchanges these two lifts, both handles must meet the fixed-point set 
of the same element $u$. It now follows that $u$ acts as a hyperelliptic 
involution (see Figure \ref{fig_strong_involution_lift}).
%
\begin{figure}[btp]
\begin{center}
\includegraphics*[width=10cm]{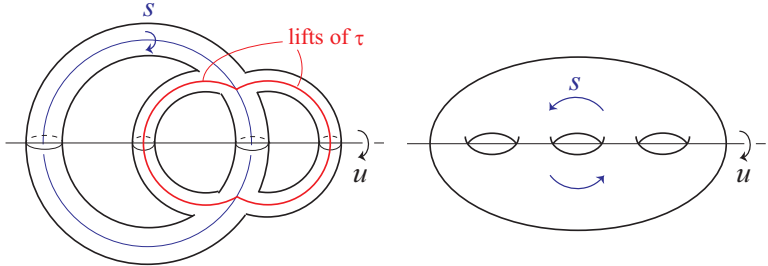}
\end{center}
\caption{$u$ acts as a hyperelliptic involution.}
\label{fig_strong_involution_lift}
\end{figure}
%

Since $u$ acts as a hyperelliptic involution, it follows that
$(M,Fix(u))/\langle u \rangle=(\S^3, K_u)$, where $K_u$ is a knot admitting a
$4$-bridge presentation. It follows that the bridge index of $K_u$ is at most
$4$. If the bridge index of $K$ is at least $5$, then $K_u$ is a $2$-twin of 
$K$, and $K$ is not determined by its double branched cover.

Assume now that the Heegaard genus of $M$ is two. Any minimal genus Heegaard
splitting of $M$ admits a hyperelliptic involution $t$. We can now repeat the
same argument as before. The involution $t$ is a deck transformation of a
double branched cover of a (necessarily) $3$-bridge knot
$K_t$. It follows that either the bridge index of $K$ is $3$ or $K$ is not
determined by its double branched cover. Note that by \cite{K} a tunnel number 
one knot of bridge index $3$ is a $(1,1)$-knot.

Assume now that $K$ is a $(1,1)$-knot. The genus-one Heegaard splitting of
$\S^3$ with respect to which $K$ is in $1$-bridge position lifts to a genus-two
Heegaard splitting for $M$, which has thus Heegaard genus at most $2$. It
follows from the discussion above that a $(1,1)$-knot of bridge index $\ge 4$
is not determined by its double branched cover.
\qed


\section{Non $\pi$-hyperbolic tunnel number one knots}
\label{s:easycases}


In this section we shall characterise the non $\pi$-hyperbolic tunnel number
one knots that are determined by their double branched covers. Recall that
tunnel number one knots are necessarily prime. A prime knot is either simple or
toroidal. Toroidal, i.e. satellite, tunnel number one knots were classified
independently by Morimoto and Sakuma \cite{MS}, and Eudave-Mu\~noz \cite{E}. 
Simple knots are either torus knots, which are all tunnel number one, or 
hyperbolic. Since tunnel number one knots are Conway irreducible, the double 
cover of a hyperbolic tunnel number one knot is atoroidal and admits a 
geometric structure. Such structure can be spherical if the knot is a 
$2$-bridge knot, Seifert fibred if the knot is a Montesinos knot or a torus
knot, or hyperbolic if the knot is $\pi$-hyperbolic. Knots of bridge index at 
most $2$ (including the trivial knot) have all tunnel number one and it is 
well-known that they are determined by their double branched covers. Tunnel 
number one Montesinos knots were classified by Klimenko and Sakuma \cite{KS}.

\subsection{Torus knots}

Let $K=T(p,q)$ be the torus knot of type $(p,q)$, with $2\le p<q$, and $p$ and 
$q$ coprime. Recall that the bridge index of $K$ is $p$, in particular if $p=2$, 
$K$ is a $2$-bridge knot and it is determined by its double branched cover. We 
can thus assume that $p\ge3$. The double branched cover $M$ of $K$ is a Seifert 
fibred manifold with base the $2$-sphere and three exceptional fibres of orders 
$(2,p,q)$ if $pq$ is odd, $(p/2,q,q)$ if $p$ is even, or $(p,p,q/2)$ if $q$ is 
even. Note that, since $p>2$ the fibration of $M$ is unique. Now $M$ admits an
involution of Montesinos type, so that $M$ is the double branched cover of a
Montesinos knot $K'$ with three rational tangles which may or may not be 
equivalent to $K$. In fact, $K'$ is hyperbolic most of the time and so, in this
case it is a $2$-twin of $K$. Remark that since the bridge index of $K$ is 
equal to $p$, we can conclude that $K$ is not determined by its double branched
cover whenever $p>3$. We only need to consider the case where $p=3$. In
this case it is known that $K'$ is hyperbolic for all $q>5$, and hence $K$ is 
not determined by its double branched cover. If $q\in\{4,5\}$ the knot is 
determined by its double branched cover. Indeed, in these two cases
$K'$ must be a simple, non hyperbolic knot. It follows that $K'$ is a torus
knot. Now since fibrations of double branched covers of torus knots are unique
we conclude that in this case $K'$ is equivalent to $K$. To conclude we only
need to observe that, by the orbifold theorem, any involution of $M$ that is
the deck transformation of a double branched cover of a knot must
preserve the Seifert fibration. If the involution induces an orientation 
preserving map of the base of the fibration, it must act as the covering 
involution for the torus knot. If the involution reverses the orientation of 
the base and fixes each exceptional fiber, it must act as a Montesinos 
involution. For the case where $(p,q)=(3,4)$, there is also a possibility that 
the involution reverses the orientation of the base, fixes the fiber of order 
$2$ and exchanges the other two of order $3$. In this case, it can be seen that 
the quotient of $M$ by the involution is a lens space of type $(3,1)$ and not 
$\S^3$ by using an argument similar to that in \cite[Proof of Lemma 3.3]{BZ} 
for instance.

\subsection{Montesinos knots}


We will exploit the classification of these knots obtained by Klimenko and
Sakuma.

\begin{Theorem}[Klimenko-Sakuma~\cite{KS}]
Let $K$ be a Montesinos knot with invariants of the form 
$$(b;(\alpha_1,\beta_1),(\alpha_2,\beta_2),\dots (\alpha_r,\beta_r))$$
where, for each $i$, we can assume that $\alpha_i$ and $\beta_i$ are coprime
and $0<\beta_i<\alpha_i$. Then $K$ is a tunnel number one knot if and only if one of 
the following condition is satisfied:
\begin{enumerate}
\item $K$ is a $2$-bridge knot (that is $r\le2$);
\item $r=3$, $\alpha_1\le\alpha_2\le\alpha_3$ and either 
\begin{enumerate}
\item $\alpha_1=2$ and $\alpha_2\alpha_3$ is odd; or
\item $\alpha_1=\alpha_2=3$, $\beta_1=\beta_2$, $\alpha_3\ge4$ is not
divisible by $3$, and the Euler number is of the form $e=\pm 1/(3\alpha_3)$. 
\end{enumerate} 
\end{enumerate}
\end{Theorem}

As in the previous subsection, we can assume that $K$ is not a $2$-bridge knot.
As a consequence its double branched cover can only be the double branched
cover of another Montesinos knot or a torus knot. Since a Seifert fibred space
with base $\S^2$ and (at most) three exceptional fibers is the double branched
cover of a unique Montesinos link, $K$ is not determined by its double branched
cover if and only if its double branched cover is also the double branched
cover of a torus knot. In order to determine which of the knots listed in the 
theorem above admit torus knots as $2$-twins, it is necessary to understand
the Seifert invariants of the double covers of torus knots. These were
determined in \cite[Theorem 1, page 13]{NR}, in fact for branched covers of any 
order.

\begin{Proposition}[\cite{NR}] \label{p:torusKcovers}
Let $K$ be the torus knot $T(p,q)$, with $p<q$ coprime. Assume that $K$ is not
a $2$-bridge knot, i.e. $p>2$. Choose integers $x$ and $y$ so that $yp+xq=-1$.
The double branched cover $M$ of $K$ is a Seifert fibred manifold with base
$\S^2$ and precisely three exceptional fibres. More precisely, the Seifert
invariants of $M$ are:
\begin{enumerate}
\item[(1)] If $p$ and $q$ are both odd, then the fibres are of type $(2,d)$,
$(p,kx)$, and $(q,ky)$, where $d$ is any odd number and $k$ is determined by
the condition $pqd=1-2k$; the Euler number is $1/(2pq)$.
\item[(2)] If $p=2k$ is even, then the fibres are of type $(k,x)$, $(q,y)$, and
$(q,y)$; the Euler number is $1/(kq)$.
\item[(3)] If $q=2k$ is even, then the fibres are of type $(p,x)$, $(p,x)$, and
$(k,y)$; the Euler number is $1/(kp)$.
\end{enumerate}  
\end{Proposition}

Note that the above proposition only takes into account right-handed torus
knots: the covers of their mirror images are obtained by changing the
orientation of the manifold, resulting in a change of sign of all Seifert
invariants.

Comparing the lists in the above results we have:

\begin{Theorem}
Let $K$ be a tunnel number one Montesinos knot. Assume that $K$ is a knot of 
type (2a) in the Klimenko-Sakuma classification, then $K$ is not determined by 
its double branched cover if and only if, up to changing the signs of its 
Seifert invariants, one has 
\begin{itemize}
\item[(2a-1)] $\alpha_2$ and $\alpha_3$ are coprime, 
$(\alpha_2,\alpha_3)\neq(3,5)$, 
$2\alpha_3\beta_2\equiv -1$ mod $\alpha_2$, $2\alpha_2\beta_3\equiv -1$ mod 
$\alpha_3$, and the Euler number is $1/(2\alpha_2\alpha_3)$; 
\item[(2a-2)] $\alpha_2=\alpha_3>4$ is odd, $4\beta_2=4\beta_3\equiv -1$ mod 
$\alpha_2=\alpha_3$, and the Euler number is $1/(2\alpha_2)$.
\end{itemize} 
Assume that $K$ is a knot of type (2b) in the Klimenko-Sakuma classification,
then $K$ is not determined by its double branched cover if and only if, up to 
changing the signs of its Seifert invariants one has 
\begin{itemize}
\item[(2b-1)] $\alpha_3\beta_1\equiv 1$ mod $3$, $3\beta_3\equiv -1$ mod 
$\alpha_3$, and the Euler number is $1/(3\alpha_3)$.
\end{itemize}
\end{Theorem}

\begin{proof}
%
Assume that $K$ is a knot of type (2a) in the Klimenko-Sakuma classification.
Then the double branched cover $M$ of $K$ is a Seifert fibered space over $\S^2$ 
with three exceptional fibers of type $(2,1)$, $(\alpha_2, \beta_2)$ and 
$(\alpha_3,\beta_3)$ where $\alpha_2\alpha_3$ is odd. Let $e$ be the Euler 
number of $M$. If $K$ admits a $2$-twin $K'$, then by arguments in the previous 
subsection on involutions of $M$ it can be seen that $K'$ must be a torus knot 
$T(p,q)$ for some $p,q$. (We may assume $p<q$.)

Assume that $p$ and $q$ are both odd. By Proposition~\ref{p:torusKcovers} (1), 
putting $d=1$, we have
$$
\alpha_2=p,\ \alpha_3=q,\ \beta_2\equiv kx\pmod{\alpha_2},\ 
\beta_3\equiv ky\pmod{\alpha_3},\ e=1/(2\alpha_2\alpha_3)
$$
where $x$, $y$ satisfy $y\alpha_2+x\alpha_3=-1$ and $k$ satisfies 
$\alpha_2\alpha_3=1-2k$.
By the first two equalities, $\alpha_2$ and $\alpha_3$ must be coprime. 
By the third equality and the conditions on $x$, $y$, $k$, we have
$$
2\alpha_3\beta_2\equiv 2\alpha_3\cdot kx= 2k\cdot x\alpha_3\equiv 1\cdot 
(-1)=-1  \pmod{\alpha_2}, 
$$
and similarly we obtain $2\alpha_2\beta_3\equiv -1\pmod{\alpha_3}$ by the 
fourth equality and the conditions on $x$, $y$, $k$.

Conversely, if the condition (2a-1) is satisfied, then it can be seen that $K$ 
admits $T(\alpha_2, \alpha_3)$ as a $2$-twin since $K$ is hyperbolic when 
$(\alpha_2,\alpha_3)\ne (3,5)$ in this case. 
When $(\alpha_2,\alpha_3)= (3,5)$, it can be seen that $K$ is a Montesinos knot 
with invariants of the form $(0; (2,1), (3,-1), (5,-1))$ and is equivalent to 
$T(3,5)$.

Assume that $p=2k$ is even. Note that $q>p\ge 4$ and must be odd. 
By Proposition \ref{p:torusKcovers} (2), we have $k=2$ and 
$$
\alpha_2=\alpha_3=q,\ \beta_2=\beta_3\equiv y\pmod{\alpha_2=\alpha_3},\ 
e=1/(2\alpha_2),
$$
where $y$ satisfies $4y\equiv -1\pmod{\alpha_2=\alpha_3}$.
Thus we obtain the condition (2a-2).

Conversely, if the condition (2a-2) is satisfied, then it can be seen that $K$ 
admits a $2$-twin $T(4, \alpha_2)$ since $K$ is hyperbolic in this case.

Assume that $q=2k$ is even. Note that $p$ must be odd. 
By Proposition~\ref{p:torusKcovers} (3), we have $k=2$, that is $q=4$, and 
$$
\alpha_2=\alpha_3=p=3,\ \beta_2=\beta_3\equiv x\pmod{3},\ e=1/2p=1/6,
$$
where $x$ satisfies $4x\equiv -1\pmod{3}$ and hence $x\equiv -1\pmod{3}$.
In this case, $K$ is a Montesinos knot with invariants of the form
$(0; (2,1), (3,-1), (3,-1))$ and is equivalent to $T(3,4)$.

Assume that $K$ is a knot of type (2b) in the Klimenko-Sakuma classification.
Then the double branched cover $M$ of $K$ is a Seifert fibered space over $\S^2$ with 
three exceptional fibers of type $(3,\beta_1)$, $(3, \beta_1)$ and 
$(\alpha_3,\beta_3)$, where $\alpha_3\ge 4$ is not divisible by $3$, and the 
Euler number $e=\pm 1/(3\alpha_3)$. If $K$ admits a $2$-twin $K'$, then again 
$K'$ must be a torus knot $T(p,q)$ for some $p,q$. (We may assume $p<q$.)
By comparing the Seifert invariants with those in 
Proposition~\ref{p:torusKcovers}, we can see that $p=3$ and $q=2k$ is even, and 
have
$$
\beta_1\equiv x\pmod{3},\ \alpha_3=k,\ \beta_3\equiv y\pmod{\alpha_3},\ 
e=1/(3k),
$$
where $x$ and $y$ satisfy $3y+2\alpha_3x=-1$. The first and the third 
equalities together with the above condition on $x$ and $y$ imply
$2\alpha_3\beta_1\equiv -1\pmod{3}$, or equivalently 
$\alpha_3\beta_1\equiv 1\pmod{3}$, and $3\beta_3\equiv -1\pmod{\alpha_3}$.

Conversely, if the condition (2b-1) is satisfied, then it can be seen that $K$ 
admits a $2$-twin $T(3, 2\alpha_3)$ since $K$ is hyperbolic (because 
$\alpha_3\ge 4$) in this case.

We also note that $K$ admits $T(p,q)$ as a $2$-twin if and only if the
mirror of $K$ admits the mirror of $T(p,q)$ as a $2$-twin. 
\end{proof}

\begin{proof}[Proof of Corollary~\ref{c:alternating}]
Recall that the Euler number $e$ is defined as $b-\sum_{i=1}^r
\beta_i/\alpha_i$. Because of the value of $e$ given in
Proposition~\ref{p:torusKcovers} it is easy to see that the rational invariants
cannot be all of the same sign.
\end{proof}

\subsection{Satellite knots}

Satellite tunnel number one knots were classified independently by Morimoto and 
Sakuma, and Eudave-Mu\~noz. It turns out that the JSJ decomposition of 
the exterior of such a knot is extremely simple and consists of just two pieces. These are the 
exteriors of a two-component $2$-bridge link $L$ (different from the trivial 
link and the Hopf link) and of a non trivial torus knot $T(p,q)$ in such a way 
that the fibre of the fibration of the torus knot is identified with the 
meridian of one of the two components of $L$. 

It was proved by Schubert that the bridge number of these knots is the product 
of the bridge number of $T(p,q)$ and the wrapping number of $K$ (see 
\cite{Sc} for a proof). Since both these numbers are at least $2$, we know that 
these knots have bridge index at least four. On the other hand,  these knots 
are known to be (1,1) by \cite[Thorem 2.1]{MS}. These facts together with the 
last assertion of Theorem~\ref{t:main} imply that these knots are never 
determined by their double branched covers. In the rest of this section, we 
give an alternative proof of this conclusion and discuss more precisely what 
$2$-twins these knots admit. 

The following fact will be useful in the sequel.

\begin{Lemma}\label{l:2bridge}
Let $L=L_1\cup L_2$ be a $2$-bridge link of type $(2\alpha,\beta)$. Consider 
$M(L_1,2)$ and let $L'$ be the lift of $L_2$ to $M(L_1,2)$. Then $M(L_1,2)$ is 
the $3$-sphere and $L'$ is a $2$-bridge knot or a $2$-component $2$-bridge link according to whether the linking number of $L_1$ and $L_2$ is odd or even. 
Moreover $L'$ is hyperbolic if and only if $\beta\not\equiv \pm 1\pmod{\alpha}$.
\end{Lemma}

\begin{proof}
The components of $L$ are both knots of bridge index $1$, so they are both
trivial. It follows at once that $M(L_1,2)$ is the $3$-sphere. Since $L$ admits
a symmetry (a $\pi$-rotation) that exchanges its components, we also see that
the lift of $L_2$ to $M(L_1,2)$ coincides with the lift of $L_1$ to $M(L_2,2)$.
Considering now a four-plat position for $L$, it is straightforward to see that 
$L'$ also admits a four-plat presentation (see Figure \ref{fig_2blink_cover}), 
so that $L'$ is a $2$-bridge link. 
Note that $\beta/2\alpha$ admits a continued fraction of the form
$$
\frac{\beta}{2\alpha}= \frac{1}{2a_1+\displaystyle\frac{1}
{a_2+\displaystyle\frac{1}{\dots+\displaystyle\frac{1}{2a_n}}}},
$$
for some odd number $n$, and that the linking number of $L_1$ and 
$L_2$ is $a_1+a_3+\cdots+a_n$. One can see from Figure \ref{fig_2blink_cover} 
that $L'$ is a knot or a $2$-component link according to whether this linking number is 
odd or even. 
As seen in Figure \ref{fig_2blink_cover}, the invariant of $L'$ is
$$
\frac{1}{a_1+\displaystyle\frac{1}{2a_2+\displaystyle\frac{1}
{\dots+\displaystyle\frac{1}{a_n}}}}=\frac{\beta}{\alpha}.
$$
The last assertion of the lemma follows from the fact that a $2$-bridge link of 
type $(\alpha,\beta)$ is hyperbolic if and only if $\beta\not\equiv \pm 
1\pmod{\alpha}$ (see \cite{Me} and also \cite[Theorem A.1]{GF}).
\end{proof}
%
\begin{figure}[btp]
\begin{center}
\includegraphics*[width=7cm]{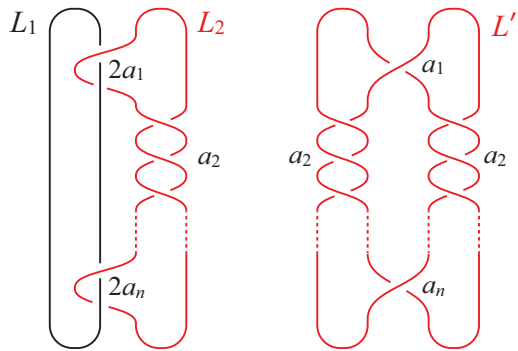}
\end{center}
\caption{The lift $L'$ of $L_2$ to $M(L_1,2)$.}
\label{fig_2blink_cover}
\end{figure}
%

Consider the double branched cover of $K$. This manifold can be obtained in the
following way. First take the double cover of the exterior of a component of
$L$ branched along the other. According to the lemma above, this cover is the 
exterior of a $2$-bridge link $L'$ with one or two components according to the 
parity of the linking number of the components of $L$. If $L'$ is a knot, we 
obtain the double branched cover of $K$ by gluing to the exterior of $L'$ the 
double (unramified) cover of the exterior of $T(p,q)$. Else, we glue on each 
boundary component of the exterior of $L'$ a copy of the exterior of $T(p,q)$. 
It now suffices to observe that the exterior of $L'$ admits an involution 
acting as a strong inversion on all components of $L'$; the double cover of the 
exterior of $T(p,q)$ admits a Montesinos type involution; $T(p,q)$ itself 
admits a strong inversion. In all cases, these involutions defined locally on 
the geometric pieces of the decomposition can be glued together to provide a 
global involution $f$ such that $(M,Fix(f))/\langle f\rangle=(\S^3,K_f)$ where 
$K_f$ is a Conway reducible (hyperbolic) knot. As such $K_f$ is necessarily a 
$2$-twin of $K$.

Observe that this shows that the $2$-twin of a tunnel number one knot need not
be a tunnel number one knot (the reader might have already remarked that this 
is also the case for some torus knots). 

On the other hand, one can prove that these knots are determined by their
double branched covers within the class of tunnel number one knots. Indeed, let
$M$ be the double branched cover of one of these knots and let $g$ be the
covering involution of a knot $K'$ such that $M(K',2)$ is diffeomorphic to $M$.
We can assume that $g$ preserves the JSJ decomposition of $M$. In particular,
either $Fix(g)$ meets some torus of the JSJ decomposition so that the $K'$ is 
Conway reducible and cannot be a tunnel number one knot, or it does not meet 
them in which case $K'$ is toroidal. It follows that if $K'$ is a tunnel number
one knot, it must belong to this family.

Now we want to show that the invariants determining $K$ can be read off the 
double cover, so that $K$ is determined within this class. Assume that $L'$
is a knot, then $L'$ is either hyperbolic or the exterior of a non
trivial torus knot whose Seifert fibration has base a disc and two exceptional 
fibres of orders $2$ and $2m+1>2$. The second piece of the JSJ decomposition of 
$M$ consists of a Seifert fibred manifold with base a disc and either two 
exceptional fibres both of odd orders (perhaps the same), or three exceptional 
fibres. 

If $L'$ is a link, the JSJ decomposition of $M$ consists of three pieces,
provided $L'$ is not the Hopf link. In any case, two of the pieces of the
decomposition are Seifert fibred with base a disc and two exceptional fibres of
coprime orders. 

Thus there is a single problematic situation, that is when $L'$ is the Hopf 
link. In this case the JSJ decomposition of $M$ consists of two copies of the
same Seifert fibred piece, i.e. the exterior of $T(p,q)$. We have to make sure 
that this double cover cannot coincide with one of the double covers where $L'$ 
is a knot. However, the two pieces of the JSJ decomposition of $M$ in the case 
where $L'$ is a knot can never be the same according to the above analysis.


\section{Two families of $\pi$-hyperbolic knots}
\label{s:two-pi}



In this section we will consider some families of $\pi$-hyperbolic tunnel number
one knots and show that generically they are not determined by their double
branched covers.

\subsection{$(1,1)$-knots of minimal Hempel distance}

Knots in $1$-bridge position with respect to genus one Heegaard 
splittings with Hempel distance at most 2 were studied by Saito in \cite{Sa}. 
If $K$ is a hyperbolic knot in this class, then Saito shows  that its exterior 
is obtained by Dehn filling a component of the exterior of a (necessarily 
hyperbolic) $2$-bridge link $L$. Note that, since both components of $L$ are 
trivial, infinitely many Dehn fillings give the exterior of a knot in $\S^3$;
moreover, by Thurston's hyperbolic Dehn filling theorem, all but finitely many
of these give a hyperbolic knot. As in the previous section, let $L'$ be the
link obtained by lifting one component of $L$ to the double branched cover of the other. As we have seen in Lemma~\ref{l:2bridge}, $L'$ is 
also hyperbolic for most choices of $L$: in this case, a sufficiently large 
surgery on a component of $L$ will result in a $\pi$-hyperbolic knot $K$.

\begin{Proposition}\label{p:hempel2}
Let $K$ be a $\pi$-hyperbolic knot constructed as above by Dehn surgery on 
some $2$-bridge link $L$ so that $L'$ is hyperbolic. If $K$ is obtained from a 
sufficiently large Dehn filling on a component of $L$, then $K$ is not 
determined by its double branched cover. 
\end{Proposition}

\begin{proof}
The double branched cover $M$ of $K$ is obtained by Dehn surgery on the 
components of the hyperbolic knot or link $L'$. If this surgery is sufficiently 
large, the core or cores of the surgery are the unique shortest geodesics of 
the hyperbolic manifold $M$. It follows that every hyperbolic isometry of $M$
leaves the core or cores of the Dehn surgery invariant and thus induces a
symmetry of the $2$-bridge knot or link $L'$; moreover different isometries
induce different symmetries. Vice-versa, each symmetry of $L'$ induces an
isometry of $M$. Besides the covering involution $s$, $M$ admits an involution
$r$ induced by a symmetry of $L'$ acting as a strong inversion on each 
component of $L'$. It is easily seen that $(M,Fix(r))/\langle r
\rangle=(\S^3,K_r)$. Since the actions of $s$ and $r$ on the shortest geodesics
of $M$ are different, these elements cannot be conjugate and we can conclude
that $K$ and $K_r$ are not equivalent, as desired.  
\end{proof}

A similar argument to the one used in the proof of the proposition above, was
used by Reni and Zimmermann to provide examples of $\pi$-hyperbolic knots with 
several $2$-twins (\cite{RZ}, see also \cite{VM}).

\begin{remark}\label{r:hempel3}
We remark that the Hempel distance of any (1,1)-splitting of $K$ constructed as 
above is at most 4 regardless of the slope of the Dehn filling, which can be 
seen as follows: Let $L_1$ and $L_2$ be the two components of $L$, and 
$V_1\cup V_2$ the genus-0 Heegaard splitting of $\S^3$ that gives the 
$2$-bridge splitting of $L$. Let $D_1$ and $D_2$ be the essential discs in 
$V_1-L$ and $V_2-L$, respectively. Let $N(L_1)$ be a regular neighborhood of 
$L_1$ disjoint from $L_2$, and let $W_1:=V_1\cup N(L_1)$ and $W_2$ the closure 
of $V_2-N(L_1)$. Then $W_1\cup W_2$ is a genus-1 Heegaard splitting of $\S^3$ 
which gives a (1,1)-splitting of $L_2$. Let $D_0$ be the co-core of the 
1-handle $N(L_1)\cap V_2$. Then $D_0$ is an essential disc of $W_1-L_2$ 
disjoint from $D_1\cup D_2$, which implies that the Hempel distance of the 
above (1,1)-splitting of $L_2$ is at most 2. Let $W_1'$ be the solid torus 
obtained from $W_1$ by applying a Dehn surgery on $L_1$. Then $W_1'\cup W_2$ 
gives a (1,1)-splitting of $K$. Note that $D_1$ and $D_2$ are essential discs 
in $W_1'-K$ and $W_2-K$, respectively, since $L_1$ is disjoint from the discs, 
and that $\partial D_0$ is an essential loop on $\partial W_1'-K$ disjoint from 
$D_1\cup D_2$. Hence, the Hempel distance of the above (1,1)-splitting of $K$ 
is at most 2. This together with the main result of \cite{T} implies that the 
Hempel distance of any other (1,1)-splitting of $K$ is at most 4.
\end{remark}

\subsection{Twisted torus knots}

Twisted torus knots were introduced by Dean in \cite{D}. These knots depend on 
four parameters and are obtained by twisting $s$ times $r$ consecutive strands 
of a torus knot $T(p,q)$. As tunnel number one knots, they can be put in 
$0$-bridge position with respect to a genus two Heegaard splitting of $\S^3$, 
by construction. Unlike tunnel number one knots, though, twisted torus knots 
can be composite or, more generally, admit non trivial $n$-tangle 
decompositions (see \cite{M1, M2}). 

We are interested in the twisted torus knots that are also tunnel number one
knots. Infinitely many examples belong to this class: for instance all knots
obtained by twisting along $r=3$ strands are tunnel number one by work of Lee
\cite{L1}. Note that this condition is sufficient but not necessary (see 
\cite{L2} or \cite[Lemma 5.4]{BTZ})

We are particularly interested in those examples that are moreover
$\pi$-hyperbolic. Lee showed that these knots are hyperbolic provided $r>1$ is
not a multiple of $p$ or $q$ and $s$ is sufficiently large. Indeed, a twisted
torus knot can be seen as the result of $1/s$-Dehn surgery along a trivial knot
$C$ encircling $r$ consecutive strands of a torus knot $T(p,q)$: Lee proved
more precisely that $C\cup T(p,q)$ is a hyperbolic link \cite{L2} under the
aforementioned hypotheses.

Of course, these knots do not need to be $\pi$-hyperbolic, however they will be so if 
we can ensure that they are tunnel number one and of bridge index at least $4$,
thus excluding the Montesinos ones. It was proved in \cite{BTZ} that the 
twisted torus knot of parameters $1<r<p<q$ has bridge index precisely $p$ 
provided that $|s|>18p$. This ensures that knots of the type that we are 
interested in do exist. 

\begin{Proposition}\label{p:ttknots}
Let $T(p,q)$ be a torus knot and $C$ a trivial knot encircling $r$ consecutive
strands of $T(p,q)$. Assume that the link $C\cup T(p,q)$ is hyperbolic and
moreover the lift $\tilde C$ of $C$ to the double branched cover of $T(p,q)$ is also a hyperbolic knot or link. Then for sufficiently large 
$s$, the ($\pi$-hyperbolic) twisted torus knot $K=K(p,q;r,s)$ obtained by 
$1/s$-Dehn surgery along $C$ is not determined by its double branched cover.
\end{Proposition} 

\begin{proof}
The proof follows the exact same lines of the proof of 
Proposition~\ref{p:hempel2}. If the surgery is sufficiently large, we can 
assume that the core or cores of the induced surgery on $\tilde C$ are the 
shortest geodesic of the double cover. Now, the double branched cover of a torus knot 
admits a Montesinos involution that induces a strong inversion of $T(p,q)$. 
Such strong inversion can be chosen (up to isotopy) so that it acts as a 
strong inversion on $C$, too. It follows that the Montesinos involution induces 
an involution of the manifold obtained by surgery on $\tilde C$ which is a deck 
transformation for a double branched cover of a knot $K'$. 
Note that the covering involutions for $K$ and $K'$ do not act in the same way
on the cores of the Dehn surgery on $\tilde C$ so they cannot be conjugate if
the cores are the shortest geodesics.
\end{proof}

The above proposition applies in particular to twisted torus knots that have
tunnel number one. Observe that among these some have bridge index $<5$
according to the discussion above.


\section{Double branched covers of $\pi$-hyperbolic knots with Heegaard splittings of
large Hempel distance}
\label{s:hempel}


This section will be devoted to the proof of Theorem~\ref{t:hempel}. Let $K$ be
a $\pi$-hyperbolic tunnel number one knot of bridge index $b\in\{3,4\}$. Let
$M$ denote its double branched cover: notice that $M$ is a hyperbolic manifold.

\subsection{The bridge index of $K$ is $3$}

Under this hypothesis $K$ is a $(1,1)$-knot, according to \cite{K}. Consider a
genus-one Heegaard splitting of the $3$-sphere with respect to which $K$ is in
$1$-bridge position. Such splitting lifts to a genus-two Heegaard splitting of
$M$, so that $M=H_1\cup_{\Sigma} H_2$, where $H_1$ and $H_2$ are genus-two
handlebodies intersecting in their common boundary $\Sigma$, a Heegaard surface
of genus $2$. As in Section~\ref{s:prfthm}, let $s$ be the involution
generating the group of deck transformations of the branched cover, and let $u$
be the lift to $M$ of a strong inversion of $K$, preserving the genus-one
Heegaard splitting of $\S^3$, chosen so that it acts as a hyperelliptic
involution of $\Sigma$ and of the splitting of $M$. We can assume that both $s$
and $u$ are isometries of the hyperbolic structure of $M$. As already remarked 
in Section~\ref{s:prfthm}, we have $(M,Fix(u))/\langle u \rangle=(\S^3,K_u)$. 
We want to show that $K_u$ is a $2$-twin of $K$. Let us assume by contradiction
that $K_u$ is equivalent to $K$, so that $s$ and $u$ are conjugate in 
$Isom^+(M)$.

Let $h\in Isom^+(M)$ such that $u=h^{-1} s h$. Consider now 
$h(H_1)\cup_{h(\Sigma)} h(H_2)$: by construction it is a genus-two Heegaard
splitting of $M$ on which $s$ acts as a hyperelliptic involution. Assume now
that the Hempel distance of $H_1\cup_{\Sigma} H_2$ and hence of
$h(H_1)\cup_{h(\Sigma)} h(H_2)$ is at least $2g+1=5$. The main result in
\cite{ST} assures that the two splittings we have are isotopic. Let $f$ be a
homeomorphism of $M$ isotopic to the identity such that $f(h(H_i))=H_i$,
$i=1,2$ and consider the homeomorphism $f s f^{-1}$: it is a hyperelliptic
involution of $\Sigma$ and the splitting. Since a surface of genus two admits a
unique hyperelliptic involution up to isotopy (see \cite{FK} for instance) we deduce that $f s f^{-1}$ and
$u$ are isotopic in $M$ and, consequently, so are $s$ and $u$. This is however
absurd since $u$ and $s$ are distinct isometries of $M$. This final
contradiction shows that $K_u$ must be a $2$-twin of $K$.

\subsection{The bridge index of $K$ is $4$}

The proof will follow the same lines of the previous case with slight
modifications. Consider a genus-two Heegaard splitting of the exterior of $K$.
As seen in Section~\ref{s:prfthm}, such splitting lifts to a genus-three 
Heegaard splitting of $M$, so that $M=H_1\cup_{\Sigma} H_2$, where $H_1$ and 
$H_2$ are genus-three handlebodies intersecting in their common boundary 
$\Sigma$ a Heegaard surface of genus $3$. Note that if this is not a minimal 
genus splitting for $M$ then $K$ is not determined by $M$, according to 
Theorem~\ref{t:main}. In this case there would be nothing to prove, so we can
assume that the Heegaard genus of $M$ is $3$. As in Section~\ref{s:prfthm}, let 
$s$ be the involution generating the group of deck transformations of the 
branched cover, and let $u$ be the lift to $M$ of a strong inversion of $K$, 
preserving the genus-two Heegaard splitting of the exterior of $K$. As remarked in Section~\ref{s:prfthm}, $u$ can be chosen so that it acts as a hyperelliptic
involution of $\Sigma$ and of the splitting of $M$. We can assume that both $s$
and $u$ are isometries of the hyperbolic structure of $M$. As already remarked
in Section~\ref{s:prfthm}, we have $(M,Fix(u))/\langle u \rangle=(\S^3,K_u)$.
If $K_u$ is a $2$-twin of $K$, again there is nothing to prove, so we can 
assume that $K_u$ is equivalent to $K$, so that $s$ and $u$ are conjugate in 
$Isom^+(M)$.

Let $h\in Isom^+(M)$ such that $u=h^{-1} s h$. Consider now
$h(H_1)\cup_{h(\Sigma)} h(H_2)$: by construction it is a genus-three Heegaard
splitting of $M$ on which $s$ acts as a hyperelliptic involution. Assume now
that the Hempel distance of $H_1\cup_{\Sigma} H_2$ and hence of
$h(H_1)\cup_{h(\Sigma)} h(H_2)$ is at least $2g+1=7$. The main result in
\cite{ST} assures that the two splittings we have are isotopic. Let $f$ be a
homeomorphism of $M$ isotopic to the identity such that $f(h(H_i))=H_i$,
$i=1,2$ and consider the homeomorphism $f s f^{-1}$: it is a hyperelliptic
involution of $\Sigma$ and the splitting. Since the genus of $\Sigma$ is $3$,
the hyperelliptic involution is not unique up to isotopy. However, since the
Hempel distance of the splitting is at least $4$, the mapping class group
$Mod(M,\Sigma)$ of homeomorphisms of $M$ preserving $\Sigma$ up to isotopies is finite according to \cite{J2}. Note that this group
contains both $u$ and $f s f^{-1}$. Now, since $Mod(M,\Sigma)$ is finite, it 
can be realised as a group of automorphisms of a complex structure on $\Sigma$:
this is a consequence of the solution to Nielsen's realisation problem, see
\cite{K1,K2}.
Such a group can contain at most one hyperelliptic element. It follows that
$f s f^{-1}=u$, so that $s$ and $u$ are isotopic in $M$. This is however
absurd since $u$ and $s$ are distinct isometries of $M$. This final
contradiction shows that $K_u$ must be a $2$-twin of $K$.


\begin{remark}\label{r:smalldistance}
We remark that the hypothesis in Theorem~\ref{t:hempel} is sufficient but not
necessary. Indeed, let $K$ be a knot as in Proposition~\ref{p:hempel2} (or as 
in Remark~\ref{r:hempel3}), which is not determined by its double branched 
cover. Let $W_1'\cup W_2$ be the genus-$1$ Heegaard splitting of $\S^3$ that 
gives the (1,1)-splitting of $K$ with Hempel distance at most $2$ as in 
Remark~\ref{r:hempel3}, and let $D_1$, $D_2$ and $D_0$ be the discs in 
$W_1'-K$, $W_2-K$ and $W_1-L_2$, respectively, as in the remark. Denote by 
$\widetilde{W_1'}$ and $\widetilde{W_2}$ the pre-images of $W_1'$ and $W_2$, 
respectively, in the double branched cover $M$ of $K$. 
Then $\widetilde{W_1'}\cup \widetilde{W_2}$ is a genus-$2$ Heegaard splitting 
of $M$. Note that $D_1$ and $D_2$ lift to essential discs $\widetilde{D_1}$ and 
$\widetilde{D_2}$ in $\widetilde{W_1'}$ and $\widetilde{W_2}$, respectively, 
and $\partial D_0$ lifts to an essential loop on the Heegaard surface which is 
disjoint from $\widetilde{D_1}\cup\widetilde{D_2}$. Hence, the Hempel distance 
of $\widetilde{W_1'}\cup \widetilde{W_2}$ is at most $2$.
\end{remark}

\paragraph{Acknowledgements}
This research was carried out during a visit of L. Paoluzzi to Nara Women's
University in 2017 funded by JSPS through FY2017 Invitational
Fellowship for Research in Japan (short term) number S17112.
L. Paoluzzi is also thankful to Nara Women's University Math Department for
hospitality during her stay.

\textsc{Department of Mathematics, Nara Women's University, Kitauoya 
Nishimachi, Nara 630-8506, Japan}

\texttt{yeonheejang@cc.nara-wu.ac.jp}

\textsc{Aix-Marseille Univ, CNRS, Centrale Marseille, I2M, UMR 7373,
13453 Marseille, France}

\texttt{luisa.paoluzzi@univ-amu.fr}

\end{document}